\documentclass[12pt,a4paper]{article}
\usepackage[utf8x]{inputenc}
\usepackage{amsmath}
\usepackage{amsfonts}
\usepackage{amssymb}
\usepackage{amsthm}
\usepackage{mathrsfs}
\usepackage{fancyhdr}
\usepackage{graphicx}
\usepackage[usenames,x11names]{xcolor}
\usepackage{arabtex}
\usepackage{framed}
\usepackage[top=2cm,bottom=2cm,margin=2cm]{geometry}
\usepackage{cancel}
\usepackage{array}   

\usepackage{hyperref}

\title{On a curious integer sequence}
\author{\sc Bakir FARHI \\
Laboratoire de Mathématiques appliquées \\
Faculté des Sciences Exactes \\
Université de Bejaia, 06000 Bejaia, Algeria \\[1mm]
\href{mailto:bakir.farhi@gmail.com}{bakir.farhi@gmail.com} \\[1mm]
\url{http://farhi.bakir.free.fr/}
}
\date{}

\def\N{{\mathbb N}}
\def\Z{{\mathbb Z}}

\def\sgn{\mathrm{sgn}} 

\def\EMdash{\leavevmode\hbox to 10.6mm{\vrule height .63ex depth -.59ex
    width 10mm\hfill}}

\theoremstyle{plain}
\numberwithin{equation}{section}
\newtheorem{thm}{Theorem}[section]

\newtheorem{prop}[thm]{Proposition}
\newtheorem{coll}[thm]{Corollary}

\newtheorem{propn}{Proposition} 

\theoremstyle{definition}

\theoremstyle{remark}

\begin{document}
\maketitle

\begin{abstract}
This note is devoted to study the recurrent numerical sequence defined by: $a_0 = 0$, $a_n = \frac{n}{2} a_{n - 1} + (n - 1)!$ ($\forall n \geq 1$). Although, it is immediate that ${(a_n)}_n$ is constituted of rational numbers with denominators powers of $2$, it is not trivial that ${(a_n)}_n$ is actually an integer sequence. In this note, we prove this fact by expressing $a_n$ in terms of the Genocchi numbers and the Stirling numbers of the first kind. We derive from our main result several corollaries and we conclude with some remarks and open problems.      
\end{abstract}


\section{Introduction and Notation}\label{sec1}

Throughout this note, we let $\N^*$ denote the set of positive integers. For a given prime number $p$ and a given positive integer $n$, we let $\vartheta_p(n)$ and $s_p(n)$ respectively denote the usual $p$-adic valuation of $n$ and the sum of base-$p$ digits of $n$. A well-known formula of Legendre (see e.g., \cite[Theorem 2.6.4, page 77]{moll}) states that for any prime number $p$ and any positive integer $n$, we have
\begin{equation}\label{eq12}
\vartheta_p(n!) = \frac{n - s_p(n)}{p - 1} .
\end{equation}
Next, we let $s(n , k)$ and $S(n , k)$ (with $n , k \in \N$, $n \geq k$) respectively denote the Stirling numbers of the first and second kinds, which can be defined as the integer coefficients appearing in the polynomial identities:
\begin{equation*}
\begin{split}
X (X - 1) \cdots (X - n + 1) & = \sum_{k = 0}^{n} s(n , k) X^k , \\
X^n & = \sum_{k = 0}^{n} S(n , k) X (X - 1) \cdots (X - k + 1) 
\end{split} ~~~~~~~~~~ (\forall n \in \N) .
\end{equation*}
This immediately implies the orthogonality relations:
\begin{equation}\label{eq13}
\sum_{k \leq i \leq n} s(n , i) S(i , k) = \sum_{k \leq i \leq n} S(n , i) s(i , k) = \delta_{n k} ~~~~~~~~~~ (\forall n , k \in \N , n \geq k) ,
\end{equation}
where $\delta_{n k}$ is the Kronecker delta. Among the many formulas related to the Stirling numbers, we mention the following (see e.g., \cite[§1.14, page 51]{com}): 
\begin{equation}\label{eq5}
\frac{\log^k(1 + x)}{k!} = \sum_{n = k}^{+ \infty} s(n , k) \frac{x^n}{n!} ~~~~~~~~~~ (\forall k \in \N) ,
\end{equation}
which is needed later on. We let finally $G_n$ ($n \in \N$) denote the Genocchi numbers which can be defined by their exponential generating function:
\begin{equation}\label{eq4}
\frac{2 x}{e^x + 1} = \sum_{n = 0}^{+ \infty} G_n \frac{x^n}{n!} . 
\end{equation} 
The famous Genocchi theorem \cite{gen} states that the $G_n$'s are all integers. It must be noted that both Stirling numbers and Genocchi numbers have combinatorial interpretations (see e.g., \cite{com,sta} for the Stirling numbers and \cite{dum,vie} for the Genocchi numbers). 

In mathematical literature, there are many examples of rational recurrent sequences that are actually integer sequences but whose integrality is not easy to prove. The Genocchi sequence ${(G_n)}_n$ is one of those sequences. Another famous example is the Somos sequences which are defined by a quadratic recurrence (see \cite{mal,som}). Furthermore, we find in Mathematics Olympiad several problems dealing with sequences defined by a rational recursion where it is asked to show their integrality (see e.g., \cite[§7]{dor}). In this note, we investigate a new interesting example of such sequences. We precisely consider the recurrent sequence ${(a_n)}_{n \in \N}$ defined by:
\begin{equation}\label{eq10}
\left\{\begin{array}{rcl}
a_0 & = & 0 , \\
a_n & = & \frac{n}{2} a_{n - 1} + (n - 1)! ~~~~ (\forall n \geq 1)
\end{array}
\right. .
\end{equation} 
Although ${(a_n)}_n$ is trivially constituted by rational numbers with denominators powers of $2$, it is not immediate that it actually consists only of integers. In the next section, we will prove this fact by expressing $a_n$ in terms of the Genocchi numbers and the Stirling numbers of the first kind. Then we derive some corollaries from our main result and conclude by mentioning a few remarks and open problems. 

\section{The results and the proofs}

Our main result is the following:

\begin{thm}\label{t1}
For all natural number $n$, we have
\begin{equation}\label{eq7}
a_n = (-1)^{n - 1} \sum_{k = 0}^{n} G_k s(n , k) .
\end{equation}
\end{thm}

To prove this theorem, we need the following intermediary results:

\begin{prop}\label{p1}
For all natural number $n$, we have
\begin{equation}\label{eq1}
a_n = \frac{n!}{2^n} \sum_{k = 1}^{n} \frac{2^k}{k} .
\end{equation}
\end{prop}

\begin{proof}
The formula is true for $n = 0$. Let us prove it for a given positive integer $n$. By definition, we have for any $k \in \N^*$:
$$
a_k - \frac{k}{2} a_{k - 1} = (k - 1)! .
$$
By multiplying the two sides of this last equality by $\frac{2^k}{k!}$, we get
$$
\frac{2^k}{k!} a_k - \frac{2^{k - 1}}{(k - 1)!} a_{k - 1} = \frac{2^k}{k} .
$$
Then by summing both sides of the last equality from $k = 1$ to $n$, we obtain (because the sum on the left is telescopic and $a_0 = 0$) that:
$$
\frac{2^n}{n!} a_n = \sum_{k = 1}^{n} \frac{2^k}{k} ,
$$
which gives the required formula. The proof is achieved.
\end{proof}

\begin{coll}\label{coll1}
The exponential generating function of the sequence ${(a_n)}_n$ is given by:
\begin{equation}\label{eq2}
\sum_{n = 0}^{+ \infty} a_n \frac{x^n}{n!} = \frac{- 2 \log(1 - x)}{2 - x} .
\end{equation}
\end{coll}

\begin{proof}
Using Formula \eqref{eq1} of Proposition \ref{p1}, we have
$$
\sum_{n = 0}^{+ \infty} a_n \frac{x^n}{n!} = \sum_{n = 1}^{+ \infty} \left(\frac{1}{2^n} \sum_{k = 1}^{n} \frac{2^k}{k}\right) x^n = \sum_{k = 1}^{+ \infty} \sum_{n = k}^{+ \infty} \frac{1}{2^n} \frac{2^k}{k} x^n = \sum_{k = 1}^{+ \infty} \frac{2^k}{k} \left(\sum_{n = k}^{+ \infty} \left(\frac{x}{2}\right)^n\right) .
$$ 
But since $\sum_{n = k}^{+ \infty} \left(\frac{x}{2}\right)^n = \left(\frac{x}{2}\right)^k \frac{1}{1 - \frac{x}{2}} = \frac{x^k}{2^k} \cdot \frac{2}{2 - x}$, we get
$$
\sum_{n = 0}^{+ \infty} a_n \frac{x^n}{n!} = \frac{2}{2 - x} \sum_{k = 1}^{+ \infty} \frac{x^k}{k} = \frac{2}{2 - x} \left(- \log(1 - x)\right) , 
$$
as required. This achieves the proof.
\end{proof}

We are now ready to prove Theorem \ref{t1}.

\begin{proof}[Proof of Theorem \ref{t1}]
Let us consider the following three functions (which are analytic on the neighborhood of zero):
$$
f(x) := \frac{- 2 \log(1 - x)}{2 - x} ~,~ g(x) := \frac{2 x}{e^x + 1} ~,~ \text{ and } h(x) := \log(1 - x) .
$$
We easily check that $f = - g \circ h$. Since in addition $h(0) = 0$ then the power series expansion of $f$ about the origin can be obtained by substituting $h$ in the power series expansion of $g$ about the origin (which is given by \eqref{eq4}) and multiplying by $(-1)$. Doing so, we get
\begin{equation}\label{eq3}
f(x) = - \sum_{k = 0}^{+ \infty} G_k \frac{(h(x))^k}{k!} = - \sum_{k = 0}^{+ \infty} G_k \frac{\log^k(1 - x)}{k!} . 
\end{equation}
Further, by substituting in \eqref{eq5} $x$ by $(- x)$, we have for any $k \in \N$:
$$
\frac{\log^k(1 - x)}{k!} = \sum_{n = k}^{+ \infty} (-1)^n s(n , k) \frac{x^n}{n!} .
$$
So, by inserting this last into \eqref{eq3}, we get
\begin{align*}
f(x) & = - \sum_{k = 0}^{+ \infty} G_k \sum_{n = k}^{+ \infty} (-1)^n s(n , k) \frac{x^n}{n!} \\
& = - \sum_{n = 0}^{+ \infty} \sum_{k = 0}^{n} (-1)^n G_k s(n , k) \frac{x^n}{n!} \\
& = \sum_{n = 0}^{+ \infty} \left[(-1)^{n - 1} \sum_{k = 0}^{n} G_k s(n , k)\right] \frac{x^n}{n!} .
\end{align*}
Comparing this with Formula \eqref{eq2} of Corollary \ref{coll1}, we conclude that:
$$
a_n = (-1)^{n - 1} \sum_{k = 0}^{n} G_k s(n , k) ~~~~~~~~~~ (\forall n \in \N) ,
$$
as required.
\end{proof}

From Theorem \ref{t1}, we derive the following important corollary:

\begin{coll}\label{coll2}
The numbers $a_n$ {\rm(}$n \in \N${\rm)} are all integers.
\end{coll}

\begin{proof}
This immediately follows from Formula \eqref{eq7} of Theorem \ref{t1} and from the fact that the Genocchi numbers and the Stirling numbers of the first kind are all integers.
\end{proof}

The integrality of the $a_n$'s gives as a consequence a nontrivial lower bound for the $2$-adic valuation of the rational numbers $\sum_{k = 1}^{n} \frac{2^k}{k}$ ($n \geq 1$). We have the following:

\begin{coll}\label{coll3}
For any positive integer $n$, we have
\begin{equation}\label{eq6}
\vartheta_2\left(\sum_{k = 1}^{n} \frac{2^k}{k}\right) \geq s_2(n) .
\end{equation}
\end{coll}  

\begin{proof}
Let $n$ be a fixed positive integer. Since $a_n \in \Z$ (according to Corollary \ref{coll2}) then we have $\vartheta_2(a_n) \geq 0$. But, by using Formula \eqref{eq1} of Proposition \ref{p1}, this is equivalent to:
$$
\vartheta_2(n!) - n + \vartheta_2\left(\sum_{k = 1}^{n} \frac{2^k}{k}\right) \geq 0 .
$$
Finally, using the Legendre formula \eqref{eq12} for the prime number $p = 2$, we have that $\vartheta_2(n!) = \frac{n - s_2(n)}{2 - 1} = n - s_2(n)$. By inserting this into the previous inequality, we get
$$
\vartheta_2\left(\sum_{k = 1}^{n} \frac{2^k}{k}\right) \geq s_2(n) ,
$$
as required.
\end{proof}

The next result express the Genocchi numbers in terms of the numbers $a_n$ ($n \in \N$) and the Stirling numbers of the second kind. It is derived from Theorem \ref{t1} and from the well-known inversion formula given by the following proposition:

\begin{propn}\label{p2}
Let ${(u_n)}_{n \in \N}$ and ${(v_n)}_{n \in \N}$ be two real sequences. Then the two following identities $(I)$ and $(II)$ are equivalent:
\begin{align}
u_n & = \sum_{k = 0}^{n} v_k s(n , k) ~~~~~~~~~~ (\forall n \in \N) , \tag{$I$} \\
v_n & = \sum_{k = 0}^{n} u_k S(n , k) ~~~~~~~~~~ (\forall n \in \N) . \tag{$II$}
\end{align}
\end{propn}

\begin{proof}
Use the orthogonality relations \eqref{eq13} (see e.g., \cite{com} or \cite{rio} for the details).
\end{proof}

\begin{coll}\label{coll4}
For any natural number $n$, we have
\begin{equation}\label{eq8}
G_n = \sum_{k = 1}^{n} (-1)^{k - 1} a_k S(n , k) .
\end{equation}
\end{coll}

\begin{proof}
It suffices to apply Proposition \ref{p2} for $u_n = (-1)^{n - 1} a_n$ and $v_n = G_n$ ($\forall n \in \N$). In view of \eqref{eq7}, Identity $(I)$ holds; so $(II)$ also, that is
$$
G_n = \sum_{k = 0}^{n} (-1)^{k - 1} a_k S(n , k) ~~~~~~~~~~ (\forall n \in \N) .
$$
Since $a_0 = 0$, the required identity follows.
\end{proof}

We end this section by providing a table of the first values of the $a_n$'s:

\bigskip

\begin{table}[!h]
\centering\begin{tabular}{c|c|c|c|c|c|c|c|c|c|c|c|c|c|}
$\boldsymbol{n}$ & $0$ & $1$ & $2$ & $3$ & $4$ & $5$ & $6$ & $7$ & $8$ & $9$ & $10$ & $11$ & $12$ \\
\hline
$\boldsymbol{a_n}$ & $0$ & $1$ & $2$ & $5$ & $16$ & $64$ & $312$ & $1812$ & $12288$ & $95616$ & $840960$ & $8254080$ & $89441280$
\end{tabular} \\
\caption{The values of the $a_n$'s for $0 \leq n \leq 12$}
\end{table}

\section{Remarks and open problems}

\subsection{Remarks}

\begin{enumerate}
\item By taking into account the facts that $G_0 = 0$, $G_1 = 1$, $G_k = 0$ for $k$ odd with $k \geq 3$, $s(n , 1) = (-1)^{n - 1} (n - 1)!$ ($\forall n \geq 1$), $\sgn(G_{2 \ell}) = (-1)^{\ell}$ ($\forall \ell \in \N^*$), and $\sgn(s(n , k)) = (-1)^{n + k}$, Formula \eqref{eq7} of Theorem \ref{t1} reduces to:
\begin{equation}\label{eq9}
a_n = (n - 1)! + \sum_{1 \leq \ell \leq \frac{n}{2}} (-1)^{\ell - 1} \left\vert G_{2 \ell} \, s(n , 2 \ell)\right\vert ~~~~~~~~~~ (\forall n \in \N) .
\end{equation}
\item By using for example a summation by parts, we easily show that $\sum_{k = 1}^{n} \frac{2^k}{k} \sim_{+ \infty} \frac{2^{n + 1}}{n}$, which gives (according to proposition \ref{p1})
\begin{equation}\label{eq11}
a_n \sim_{+ \infty} 2 \cdot (n - 1)! .
\end{equation}
\end{enumerate}

\subsection{Open problems}

\begin{enumerate}
\item Find an alternative proof that ${(a_n)}_n$ is an integer sequence without use the Genocchi theorem (according to which the $G_n$'s are all integers). Such a proof will give us a new proof of the Genocchi theorem (through Formula \eqref{eq8} of Corollary \ref{coll4}).
\item Find a generalization of Corollary \ref{coll3} to other prime numbers $p$ other than $p = 2$. Notice that the generalization that might immediately come to mind:
$$
\vartheta_p\left(\sum_{k = 1}^{n} \frac{p^k}{k}\right) \geq s_p(n) 
$$
is false for $p > 2$ (take for example $n = 2$).
\item Find a combinatorial interpretation for the numbers $a_n$ ($n \in \N$). This seems possible since the different relations in which $a_n$ appears contain numbers that have all combinatorial meanings (as $n!$, $G_n$, $s(n , k)$, and $S(n , k)$). Such an interpretation will immediately show again that the $a_n$'s are integers.
\end{enumerate}

\rhead{\textcolor{OrangeRed3}{\it References}}

\end{document}